\documentclass[12pt,twoside,reqno]{amsart}
\usepackage{amsfonts}
\usepackage{amssymb}
\usepackage[notcite,notref]{
}

\numberwithin{equation}{section}

\newtheorem{teo}{Theorem}[section]

\newtheorem{lem}[teo]{Lemma}

\theoremstyle{definition}

\numberwithin{equation}{section}

\def\a{\alpha}
\def\b{\beta}
\def\l{\lambda }

\def\o{\omega}
\def\R{\mathbb{R}}
\def\N{\mathbb{N}}

\def\d{\delta}
\def\e{\varepsilon}
\def\t{\theta}
\def\f{\varphi}
\def\s{\sigma}
\def\mes{\operatorname{mes}}

\def\mes{\operatorname{mes}}

\begin{document}


\title[Rearrangement estimates and limiting embeddings]{Rearrangement estimates and limiting embeddings for anisotropic Besov spaces }

\author[V.I. Kolyada]{V.I. Kolyada}
\address{Department of Mathematics\\
Karlstad University\\
Universitetsgatan 1 \\
651 88 Karlstad\\
SWEDEN} \email{viktor.kolyada@gmail.com}



\begin{abstract}

The paper is dedicated to the study of embeddings of the anisotropic Besov spaces $B_{p,\t_1,\dots,\t_n}^{\b_1,\dots,\b_n}(\R^n)$ into Lorentz spaces. We find the sharp asymptotic behaviour of embedding constants when some of the exponents $\b_k$ tend to 1 ($\b_k<1$). In particular, these results give an extension of the estimate proved by Bourgain, Brezis, and Mironescu for isotropic Besov spaces. Also, in the limit, we obtain a link with some known embeddings of anisotropic Lipschitz spaces.

One of the key results of the paper is an anisotropic type estimate of rearrangements in terms of partial moduli of continuity.

\end{abstract}


\dedicatory{Dedicated to Professor O.V. Besov on the occasion of his 90th birthday}

\maketitle

\date{}

\maketitle

\section{Introduction}

In the classical Function Spaces Theory, embeddings of Sobolev spaces and Besov spaces represented two different directions.
The link between them was explicitly exhibited in the following results obtained by Bourgain, Brezis and Mironescu \cite{BBM1}, \cite{BBM2}. First, it was proved in
\cite{BBM1} (see also \cite{Br})  that there exists a limiting
relation between Sobolev and Besov norms, that is,  for any $f\in
W^1_p({\mathbb R}^n)$
 $(1\le p<\infty)$
\begin{equation}\label{bbm1}
\lim_{\a\to 1-0}(1-\a)\int_{\R^n}\int_{\R^n}\frac{|f(x)-f(y)|^p}{|x-y|^{n+\a p}}\,dx\,dy\asymp\|\nabla f\|_{p}^p.
\end{equation}
The main result of \cite{BBM2} concerns a well known Sobolev-type
embedding $B_p^\a\subset L^{p^*},p^*=np/(n-\a p),$ proved in
the late sixties independently by several authors (for the
references, see \cite[Sect. 10]{K4}). In \cite{BBM2}, the sharp asymptotics of the best constant
as $\a\to 1$ in a related inequality were found, namely, if $1/2\le
\a<1$ and $1\le p<n/\a,$ then for any $f\in B_p^\a({\mathbb R}^n)$,
\begin{equation}\label{bourg}
\|f\|^p_{L^{p^*}}\le c_n\frac{1-\a}{(n-\a p)^{p-1}}
\int_{\R^n}\int_{\R^n}\frac{|f(x)-f(y)|^p}{|x-y|^{n+\a p}}\,dx\,dy.
\end{equation}
In view of
(\ref{bbm1}), the classical Sobolev inequality
$$
\|f\|_{L^{np/(n-p)}}\le c\|\nabla f\|_p,\quad 1\le p < n,
$$
can be considered as a limiting case of (\ref{bourg}).
  Note that the
proof of (\ref{bourg}) in \cite{BBM2} was quite complicated.
Afterwards Maz'ya and Shaposhnikova \cite {MS} gave
 a simpler proof  of this result.

 It was  observed in \cite{KL} that  inequality
  (\ref{bourg})  can be directly
derived from the following rearrangement estimate obtained in \cite{K1}:
\begin{equation}\label{estim6}
\int_{\d^n}^{\infty}t^{-p/n}\int_0^t\big(f^*(u)-f^*(t)\big)^pdu
\frac{dt}{t}\le c_{p,n}\left(\frac{\o(f,\d)_p}{\d}\right)^{p},
\end{equation}
($f\in L^p({\mathbb
R}^n),~~1\le p<\infty, ~~n\in \N$, and $\d>0$).
 Moreover, it was shown  in \cite{KL}
that the left-hand side of (\ref{bourg})  can be replaced by a stronger Lorentz norm $\|f\|_{L^{p^*,p}}.$

Later on, limiting embedding theorems were studied by many authors  (see \cite{KMX}, \cite{K8}, \cite{Kcap},  \cite{LS}, \cite{Mil},  \cite{Tri}).

In this paper we prove limiting embeddings for anisotropic Besov spaces defined by conditions on {\it partial} moduli of continuity. Our proof  is based upon
 the following anisotropic analogue of (\ref{estim6}), which is one of the main results of the paper.

 \begin{teo}\label{Anis} Let $f\in L^p(\R^n),\,\, 1\le p<\infty$. There exist positive functions $u_j$ ($j=1,\dots, n)$ on $\R_+$ such that
\begin{equation*}
\prod_{j=1}^n u_j(t)\le t\quad\mbox{for all}\quad t>0
\end{equation*}
and for any $j=1,\dots n$ and any $h>0$
\begin{equation*}
\int_{\Omega_j(h)}\frac{[f^*(t)-f^*(2t)]^p}{u_j(t)^p}\,dt\le c\left(\frac{\o_j(f;h)_p}{h}\right)^p,
\end{equation*}
where $\Omega_j(h)=\{t>0: u_j(t)\ge h\}.$ Besides,
\begin{equation*}
\sup_{t\in \Omega_j(h)}t^{1/p}[f^*(t)-f^*(2t)]/u_j(t)\le c \o_j(f;h)_p/h.
\end{equation*}
\end{teo}

\vskip 6pt

 Denote by $B_{p,\t;j}^{\a}(\R^n)$  the class of functions $f\in L^p(\R^n)$ satisfying Besov--type condition of the order $0<\a<1$
with respect to variable $x_j$ (see Section 2 below).

Concerning Besov spaces, our main result is the following estimate of the Lorentz norm.

\begin{teo}\label{PRINCIPAL} Let $1\le p<\infty, ~~p\le \t_j\le
\infty,$ and $0<\beta_j<1$ for all $j=1,...,n,$ where $n\ge 2.$ Set
$$
\beta=n\left(\sum_{j=1}^n\frac1{\beta_j}\right)^{-1}, \quad
\t=\frac{n}{\beta}\left(\sum_{j=1}^n\frac1{\beta_j\t_j}\right)^{-1}.
$$
Assume that $1\le p<n/\beta.$  Let $q=np/(n-\beta p).$ Then there exists a constant $c>0$ such that for
any function
$$
f\in B_{p,\t_1,\dots,\t_n}^{\b_1,\dots,\b_n}(\R^n)=\bigcap_{j=1}^n B_{p,\t_j;j}^{\beta_j}(\R^n)
$$
we have that
\begin{equation}\label{main2}||f||_{{q,\t}}\le c
\prod_{j=1}^n\left[(1-\beta_j)^{1/\t_j}||f||_{b_{p,\t_j;j}^{\beta_j}}\right]^{\beta/(n\beta_j)}.
\end{equation}
\end{teo}

\vskip 6pt

Observe that we confine ourselves to the case when $\t_j\ge p$ for all $j$. The case when for some $j$ we have  $\t_j<p$,  remains open.

 In the case $\t_j=p$ $(j=1,\dots,n)$ inequality (\ref{main2}) was proved in \cite{K8} using embeddings of mixed norm spaces.
Further,  if $\t\le q$, then (\ref{main2}) follows from a stronger inequality
(involving iterated rearrangements) proved in \cite{K9} (see Section 5 below). In the present paper we find a complete proof of Theorem \ref{PRINCIPAL}
based upon Theorem \ref{Anis}.

As in (\ref{bourg}), the most important role in  inequality (\ref{main2}) is played by the
 factors $(1-\b_j)^{1/\t_j}$. Actually, without these factors this inequality   is known (see \cite[\S 18]{BIN}, \cite{Gol}, \cite{K4},  \cite{Pe}, \cite{Perez}); moreover, in such form
  it can be  obtained  for all $1\le \t_j\le \infty$, using the "weak type" rearrangement estimates (see \cite[Section 10]{K4}, \cite{Perez}).

Inequality (\ref{main2})
gives a sharp estimate of the  Lorentz norms of
the function $f$ in terms of its Besov norms.
However, the problem can be formulated in a different way: given a
function $f\in L^p(\R^n),$ find sharp estimates of its "size" in
terms of partial moduli of continuity of $f$ (partly, such estimates are contained in Theorem \ref{Anis}).
 This  problem
was posed by  Ul'yanov \cite{U}. It is  more general and may lead
to essentially sharper results.  We refer to our survey paper \cite[Ch. 7]{K9} for the discussion of this question.

Theorem \ref{PRINCIPAL} provides the link between embeddings of Besov-Nikol'skii spaces and embeddings of Lipschitz type spaces with limiting exponent. If $1\le p<\infty$ and $0<\a_k\le 1\,\, (k=1,\dots,n)$ then the Lipschitz space $\Lambda_p^{\a_1,\dots,\a_n}(\R^n)$ is defined as the class of all functions $f\in L^p(\R^n)$ for which the partial moduli of continuity satisfy the conditions
$$
\o_k(f;\d)_p=O(\d^{\a_k})\quad (k=1,\dots,n)
$$
(see \cite{K5}, \cite{K4},  \cite{K9}). It was  shown in \cite[p. 75]{K9} that some known embeddings for the spaces $\Lambda_p^{\a_1,\dots,\a_n}(\R^n)$
can be obtained as limiting case of
Theorem \ref{PRINCIPAL}.

\section{Definitions and auxiliary propositions}

Throughout this paper we shall use the following notation. Let  $x=(x_1,\ldots,x_n).$  Denote by $\widehat{x}_k$
the $(n-1)$-dimensional vector obtained from the $n$-tuple $x$ by
removal of its $k$-th coordinate.
 We shall write
 $x=(x_k,\widehat x_k).$

\subsection{Nonincreasing rearrangements and Lorentz spaces}

Let $f$ be a measurable function $f$ on $\mathbb{R}^n.$ {\textit The distribution function} of $f$ is defined by the equality
\begin{equation*}
\lambda_f (y)= | \{x \in \mathbb{R}^n : |f(x)|>y \}|, \quad y\ge 0.
\end{equation*}
The function $f$ is said to be {\it rearrangeable} if $\l_f(y)<\infty$ for all $y>0.$ We denote by $S_0(\R^n)$ the class of all rearrangeable functions on $\R^n.$

{\it{A nonincreasing rearrangement}} of a function $f\in S_0(\R^n)$
 is a nonnegative and nonincreasing function $f^*$ on
$\mathbb{R}_+ $ which is equimeasurable with $|f|$, that is, $\l_{f^*}=\l_f.$
We shall assume in addition that the rearrangement $f^*$ is left-continuous on $\R_+.$ Under this condition it is defined
uniquely by
 the equality (see \cite[p. 32]{ChR}, \cite{Ksib})
\begin{equation}\label{rearrangement}
f^*(t) = \sup_{|E|=t} \inf_{x \in E} |f(x)|,\quad 0<t<\infty.
\end{equation}

 Let $0<p,r<\infty.$ A rearrangeable
function $f$ on $\R^n$ belongs to the Lorentz space
$L^{p,r}(\mathbb{R}^n)$ if
\begin{equation*}
\|f\|_{L^{p,r}}=\|f\|_{p,r} = \left(\int_0^\infty \left( t^{1/p} f^*(t)
\right)^r\, \frac{dt}{t} \right)^{1/r} < \infty.
\end{equation*}
For $0<p<\infty,$ the space $L^{p,\infty}(\mathbb{R}^n)$ is
defined as the class of all rearrangeable functions $f$ on $\R^n$ such that
$$
\|f\|_{p,\infty} = \sup_{t>0}t^{1/p} f^*(t)<\infty.
$$

We have that $||f||_{p,p}=||f||_p.$ For a
fixed $p$, the Lorentz spaces $L^{p,r}$ strictly increase as the secondary
index $r$ increases; that is, the strict embedding $L^{p,r}\subset L^{p,s}~~~(r<s)$ holds (see \cite[Ch. 4]{BS}).

\subsection{Moduli of continuity}
Let $f\in
L^p({\mathbb R}^n)$ and $k\in\{1,...,n\}.$ Set for $h\in \R$
$$
I_k(f;h)_p=\left(\int_{\mathbb R^n}
|f(x+he_k)-f(x)|^p~dx\right)^{1/p}
$$
($e_k$ is the $k-$th unit coordinate vector).
The partial modulus of
continuity of $f$ in $L^p$ with respect to $x_k$ is defined by
$$
\o_k(f;\delta)_p =\sup_{|h|\le \d} I_k(f;h)_p,\quad \d\ge 0.
$$
 Of course, in the case $n=1$ we write simply $\o(f;\delta)_p.$

 We shall use the well known    inequality
\begin{equation}\label{int}
\o_k(f;\d)_p\le \frac{3}{\d}\int_0^\d I_k(f;h)_p\,dh \quad (\d>0),
\end{equation}
which easily follows from the subadditivity of $I_k(f;h)$ (see, e.g.,  \cite[Lemma 2.1]{K8}).

We shall call {\it modulus of continuity} any non-decreasing,
continuous and bounded function $\o(\d)$ on $[0,+\infty)$ which
satisfies the conditions
\begin{equation}\label{d1}
 \o(\d+\eta)\le \o(\d)+\o(\eta),
\quad \o(0)=0.
\end{equation}

It is easy to see that for any $f\in L^p({\mathbb R}^n)$ the functions
$\o_j(f;\d)_p$ are moduli of continuity.

If $\o$ is a modulus of continuity, then by (\ref{d1})
\begin{equation}\label{d2}
\o(2^k\d)\le 2^k\o(\d)\quad\mbox{for any}\quad k\in\N\quad\mbox{and any}\quad \d>0.
\end{equation}
It follows that
\begin{equation}\label{d3}
\o(\mu)/\mu\le 2\o(h)/h \quad\mbox{for}\quad 0<h<\mu.
\end{equation}

  Observe that for a modulus of continuity $\o$ the
function $\o(\d)/\d$ may not be monotone. However,  for
any modulus of continuity $\o$
\begin{equation}\label{d4}
\lim_{\d\to 0+}\frac{\o(\d)}{\d}=\sup_{\d>0}\frac{\o(\d)}{\d}.
\end{equation}
Indeed, denote the latter supremum by $\s$. If $\s=\infty,$ then  (\ref{d4}) follows immediately from (\ref{d2}).
Let $\s<\infty$ and let $0<\e<\s/2.$ There exists $\d_0>0$ such that $\o(\d_0)>(\s-\e)\d_0.$ Set $\d_k=2^{-k}\d_0,$
$k\in\N.$ Then by (\ref{d2}) $\o(\d_k)>(\s-\e)\d_k.$ Using (\ref{d1}), we have for any $\d_{k+1}<\d<\d_k$
$$
\o(\d_k)\le \o(\d)+ \o(\d_k-\d).
$$
Since $\o(\d_k-\d)\le \s(\d_k-\d),$ we obtain that $\o(\d)>(\s-2\e)\d.$ This implies (\ref{d4}).

Let $f\in L^1_{\operatorname{loc}}(\R^n))$, $j\in \{1,\dots, n\},$ and $h>0.$ We consider the Steklov means
$$
f_{h,j}(x)=\frac1h\int_0^h f(x+ue_j)\,du.
$$
It is easy to see that for any $f\in L^p(\R^n)\,\,(1\le p<\infty)$
\begin{equation}\label{steklov1}
||f-f_{h,j}||_p\le \o_j(f;h)_p
\end{equation}
and
\begin{equation}\label{steklov2}
\left\|\frac{\partial{f_{h,j}}}{\partial{x_j}}\right\|_p\le \frac{\o_j(f;h)_p}{h}.
\end{equation}

In what follows we shall use estimates of
the $L^p-$modulus of continuity of the rearrangement $f^*$ in
terms of the modulus of continuity of a given function $f.$

The first sharp results in this direction were obtained  independently
by  Oswald \cite{Osw1} and  Wik \cite{Wik1}. In particular, it was proved, that
 for any $f\in L^p[0,1], ~~1\le p<\infty,$
\begin{equation}\label{o-w2}
\o(f^*;\d)_p\le 2\o(f;\d)_p, \quad 0\le \d\le \frac12.
\end{equation}

\vskip 8 pt

\subsection{Besov spaces}

Let $0<\a<1, ~k\in\{1,...,n\},$ and $1\le p,\t<\infty.$ The Besov
 space $B^\a_{p,\t;k}({\mathbb R}^n)$ is defined as the space of
 all functions
 $f\in L^p(\mathbb {R}^n)$ such that
$$
\|f\|_{b^\a_{p,\t;k}}=
\left(\int_0^{\infty}\left(t^{-\a}\o_k(f;t)_p\right)^\t
\frac{dt}{t}\right)^{1/\t}<\infty.
$$
For $\t=\infty$, the space $B^\a_{p,\infty;k}({\mathbb R}^n)$ consists of
 all functions
 $f\in L^p(\mathbb {R}^n)$ such that
 $$
\|f\|_{b^\a_{p,\infty;k}}=
\sup_{\d>0}\frac{\o_k(f;\d)_p}{\d^\a}<\infty.
$$
Denote also $B^\a_{p,p;k}= B^\a_{p;k}.$

It is easy to see that
\begin{equation}\label{predel}
\|f\|_{b^\a_{p,\infty;k}}=\lim_{\t\to +\infty} \|f\|_{b^\a_{p,\t;k}}.
\end{equation}

Let $0<\a_k<1,\,\,1\le \t_k\le \infty ~~(k=1,...,n),$ and $1\le p<\infty.$ Then we set
 $$
 B^{\a_1,...,\a_n}_{p,\t_1,\dots,\t_n}({\mathbb R}^n)
=\bigcap_{k=1}^nB^{\a_k}_{p,\t_k;k}({\mathbb R}^n).
$$

Observe that in these definitions and notations we follow Nikol'skii's book \cite{Nik}.

As we have already mentioned in the Introduction,
Bourgain, Brezis and Mironescu \cite{BBM1}
found a limiting relation between Sobolev and Besov norms (see (\ref{bbm1})).

For the partial Besov norms we have the following statement.
\begin{lem}\label{K-L} Let $f\in B^\a_{p,\t;k}(\R^n)$, where $1\le p,\t<\infty,$ $0<\a<1$, and  $1\le k\le n$. Then
\begin{equation}\label{k-l}
\lim_{\a\to
1-0}(1-\a)^{1/\t}\|f\|_{b^{\a}_{p,\t;k}}=\left(\frac{1}{\t}\right)^{1/\t}\sup_{\d>0}
\frac{\o_k(f,\d)_p}{\d}\,.
\end{equation}
\end{lem}

For the proof, see \cite{KL}, \cite{K9}; actually, in view of (\ref{d4}), it easily follows by standard arguments.

\vskip 8pt

\subsection{Class $\mathcal M_{dec}(\R^n_+)$ and iterated rearrangements}

In this subsection we shall consider functions defined on $\R^n_+.$

Denote by $\mathcal{M}_{dec}(\R^n_+)$ the class of all nonnegative functions on $\R^n_+$ which are nonincreasing in each variable.

Let $f$ be a rearrangeable function on $\R^n$ and let $1\le k \le n.$ We fix
$\widehat{x}_k\in \Bbb R^{n-1} $ and consider the
$\widehat{x}_k$-section of the function $f$
$$
f_{\widehat{x}_k}(x_k)=f(x_k,\widehat{x}_k), \quad x_k\in \Bbb R.
$$
For almost all
 $\widehat{x}_k\in \Bbb R^{n-1}$ the function
$f_{\widehat{x}_k}$ is rearrangeable on $\R$. We set
$$
\mathcal R_kf(u,\widehat{x}_k)=f_{\widehat{x}_k}^*(u), \quad u\in
\Bbb R_+.
$$
Stress that the $k$-th argument of the function $\mathcal R_kf$ is
equal to $u.$ The function
 $\mathcal R_kf$ is defined almost everywhere on  $\Bbb R_+\times \Bbb
R^{n-1}$; we call it {\it {the  rearrangement of $f$ with respect to the
$k$-th variable}.} It is easy to show that
$\mathcal R_kf$ is a measurable function equimeasurable with
$|f|$.
Let $\mathcal P_n$ be the collection of all permutations
 of the set $\{1,\dots,n\}$. For each
$\sigma=\{k_1,\dots,k_n\}\in \mathcal P_n$ we call the function
$$
\mathcal R_\sigma f(t)=\mathcal R_{k_n}\cdots\mathcal R_{k_1}f(t), \quad
t\in \Bbb R_+^n,
$$
the $\mathcal R_\sigma$-rearrangement of  $f$. Thus, we obtain
$\mathcal R_\sigma f$  by "rearranging"  $f$ in non-increasing
order successively with respect to the variables
$x_{k_1},\dots,x_{k_n}$.
The rearrangement
$\mathcal R_\sigma f$  is defined on $\R_+^n.$ It is nonnegative, nonincreasing in each variable,
and equimeasurable
with  $|f|$  (see \cite{{BKPS}, Bl,  K9, K2012}).

The next  lemma follows by known reasonings (see \cite{Osw_dis}). We outline the proof, because this lemma is essentially used  in the sequel.

\begin{lem}\label{MODULI} Let $f\in L^p(\R^n)~~(1\le p<\infty).$
Then for each permutation $\sigma\in \mathcal P_n$ and each
$k=1,...,n$
\begin{equation}\label{moduli}
\o_k(\mathcal R_\sigma f;\d)_p\le 3^n\o_k(f;\d)_p \quad\text{for
all}\quad \d\ge 0.
\end{equation}
\end{lem}
\begin{proof} We apply induction. For $n=1$ (\ref{moduli}) follows from (\ref{o-w2}). Assume that $n\ge 2$ and (\ref{moduli})
holds for the dimension $n-1$. Let  $\s=\{k_1,\dots,k_n\}$ and $\widehat \s_n=\{k_1,\dots,k_{n-1}\}.$ Without loss of generality we may assume that $k_n=n$.
Set
$$
\f(t)=\mathcal R_\s f(t),\quad \psi(\widehat t_n, x_n)=\mathcal R_{\widehat \s_n}f(\widehat t_n, x_n),\,\,\,(t=(\widehat t_n, t_n)\in\R_+^n,\,\, x_n\in\R),
$$
$$
I_k(h)=\left(\int_{\R_+^n}|\f(t)-\f(t+he_k)|^p\,dt\right)^{1/p}\quad (h>0).
$$
 First we assume that $1\le k\le n-1.$ Fix $\widehat t_n\in\R^{n-1}_+$. Then
$\f(\widehat t_n, t_n)$ is the rearrangement of $\psi(\widehat t_n, x_n)$ with respect to $x_n$. By nonexpansivity of rearrangement (see \cite[p. 83]{LL}),
\begin{equation*}
\int_{\R_+}|\f(\widehat t_n, t_n)-\f(\widehat t_n+he_k, t_n)|^p\,dt_n\le
\int_{\R}|\psi(\widehat t_n, x_n)-\psi(\widehat t_n+he_k, x_n)|^p\,dx_n.
\end{equation*}
By our inductive hypothesis and (\ref{int}), for a fixed $x_n$
$$
\begin{aligned}
&\int_{\R_+^{n-1}}|\psi(\widehat t_n, x_n)-\psi(\widehat t_n+he_k, x_n)|^p\,d\widehat t_n\\
&\le 3^{(n-1)p}\o_k(f_{x_n};h)_p^p\le \frac{3^{np}}{h}\int_0^h\int_{\R^{n-1}}|f(x)-f(x+ue_k)|^p\,d\widehat x_n\,du,
\end{aligned}
$$
where $f_{x_n}(\widehat x_n)=f(x_n,\widehat x_n).$ Thus,
$$
I_k(h)^p\le \frac{3^{np}}{h}\int_0^h\int_{\R^n}|f(x)-f(x+ue_k)|^p\,dx\,du,
$$
and we obtain (\ref{moduli}) for $1\le k<n.$

Let now $k=n$. By (\ref{o-w2}) and (\ref{int}), for a fixed $\widehat t_n$ we have
$$
\int_{\R_+}|\f(\widehat t_n,t_n)-\f(\widehat t_n,t_n+h)|^p\,dt_n\le \frac{6^p}{h}\int_0^h\int_\R |\psi(\widehat t_n, x_n)-\psi(\widehat t_n, x_n+u)|^p\,dx_n\,du.
$$
Integrating with respect to  $\widehat t_n$, we obtain
$$
I_k(h)^p\le \frac{6^{p}}{h}\int_0^h\int_{\R_+^{n-1}}\int_\R|\psi(\widehat t_n, x_n)-\psi(\widehat t_n, x_n+u)|^p\,dx_n\,d\widehat t_n\,du.
$$
The nonexpansivity of rearrangement implies that for fixed $x_n$ and $u$
$$
\int_{\R_+^{n-1}}|\psi(\widehat t_n, x_n)-\psi(\widehat t_n, x_n+u)|^p\,d\widehat t_n\le
\int_{\R^{n-1}}|f(\widehat x_n, x_n)-f(\widehat x_n, x_n+u)|^p\,d\widehat x_n.
$$
These inequalities imply (\ref{moduli}) for $k=n.$
\end{proof}

\vskip 6pt

\subsection{Projections and sections}

Let $E\subset \R^n.$ For every $k=1,...,n,$ denote by $\Pi_k(E)$
the orthogonal projection of $E$ onto the coordinate hyperplane
$x_k=0.$ Further, if $\widehat{x}_k\in \R^{n-1},$ then  by
$E(\widehat{x}_k)$ we denote the $\widehat x_k-$section of $E,$
$$
E(\widehat{x}_k)=\{x_k\in \R: (x_k,\widehat{x}_k)\in E\}.
$$

If a set $A\subset \R^n$ is contained in some $k-$dimensional plane in $\R^n$ and is measurable with respect to the $k-$dimensional Lebesgue measure, then its
Lebesgue measure
will be
denoted by $\mes_k A.$ For $k=n$ we denote the $n$-dimensional measure also by $|A|$.

Considering sets in $\R^n$, we assume that they are of type $F_\s.$ Then their projections and sections also are of type $F_\s$ and therefore are Lebesgue measurable in the corresponding spaces.

We shall use the following lemma.

\begin{lem}\label{projections} Let $E\subset \R^n$ be a set of type $F_\s$ with $|E|=t, \,\,0<t<\infty.$ Then there exist  sets
\begin{equation}\label{project0}
E_0\supset E_1\supset E_2\supset\dots\supset E_n\quad(E_0=E)
\end{equation}
of type $F_\s$ with $|E_j|=2^{-j}t$ such that for any $j=1,\dots, n$ and any $F_\s$-set $A\subset E_{j-1}$ with
$|A|\ge 2^{-j}t$ it holds the inequality
\begin{equation}\label{project}
\mes_{n-1} \Pi_j(E_j)\le \mes_{n-1} \Pi_j(A).
\end{equation}
\end{lem}
\begin{proof} For $y\in \R^{n-1},$  set $\s_1(y)=\mes_1 E(y)$ if $y\in \Pi_1(E)$ and
$\s_1(y)=0$ otherwise. We have
$$
t=|E|=\int_{\R^{n-1}}\s_1(y)\,dy=\int_0^\infty \s_1^*(u)\,du.
$$
There exists $0<\mu<\infty$ such that
$$
\int_0^\mu \s_1^*(u)\,du=t/2.
$$
Clearly, $\s_1^*(\mu)>0.$ Let
$$
P=\{y\in \Pi_1(E):\s_1(y)\ge\s_1^*(\mu)\}.
$$
Then $\mes_{n-1}P\ge \mu.$ Besides,
$$
\mes_{n-1}\{y\in \Pi_1(E):\s_1(y)>\s_1^*(\mu)\}\le \mu.
$$
Therefore there exists a measurable set  $P_1\subset P$   such that
$\mes_{n-1} P_1=\mu$
and $\s_1(y)\le\s_1^*(\mu)$ for  all $y\not\in P_1.$ Further, let
$$
E_1=\{x=(x_1,\widehat x_1)\in E: \widehat x_1\in P_1\}.
$$
Then $\Pi_1(E_1)=P_1$ and
$$
|E_1|=\int_{P_1}\s_1(y)\,dy=\int_0^\mu \s_1^*(u)\,du=t/2.
$$
Let $A\subset E$ be a set of type $F_\s$ with $|A|\ge t/2$ and let $Q=\Pi_1(A).$ We state that
$$
\mes_{n-1} Q\ge \mes_{n-1} P_1.
$$
Indeed, otherwise we would have that
$$
\mes_{n-1}(P_1\setminus Q)>\mes_{n-1}(Q\setminus P_1).
$$
But $\s_1(y)\le \s_1^*(\mu)$ for all $y\in Q\setminus P_1$ and $\s_1(y)\ge \s_1^*(\mu)$ for  all $y\in P_1.$
Thus,
$$
\int_{P_1\setminus Q}\s_1(y)\,dy>\int_{Q\setminus P_1}\s_1(y)\,dy.
$$
This implies that
$$
\begin{aligned}
&|A|=\int_Q\mes_1A(y)\,dy\le \int_Q\s_1(y)\,dy=\int_{Q\cap P_1}\s_1(y)\,dy+\int_{Q\setminus P_1}\s_1(y)\,dy\\
&<\int_{Q\cap P_1}\s_1(y)\,dy+\int_{P_1\setminus Q}\s_1(y)\,dy=\int_{P_1}\s_1(y)\,dy=t/2,
\end{aligned}
$$
which contradicts the condition $|A|\ge t/2$.

Next, we consider the projection $\Pi_2(E_1)$ and $\widehat x_2$--sections of $E_1$, and apply the same reasonings to obtain the set $E_2.$
Continuing this process and using induction, we obtain the sets (\ref{project0})
satisfying  (\ref{project}).

\end{proof}

A very useful control of projections is provided by
the following Loomis--Whitney lemma \cite{LW}.
\begin{lem}\label{Loomis} Let $E\subset \R^n$ be a set of type $F_\s.$
Then
\begin{equation}\label{lw}
(\mes_nE)^{n-1}\le \prod_{k=1}^n\mes_{n-1}\Pi_k(E).
\end{equation}
\end{lem}

\vskip 6pt

\section{Rearrangement estimates}

In this section  we prove the following anisotropic version of the inequality (\ref{estim6}).

\begin{teo}\label{Huvud} Let $f\in L^p(\R^n),\,\, 1\le p<\infty,\,\, n\ge 2$. There exist positive functions $u_j$ ($j=1\dots n)$ on $\R_+$ such that
\begin{equation}\label{huvud1}
\prod_{j=1}^n u_j(t)\le t\quad\mbox{for all}\quad t>0
\end{equation}
and for any $j=1,\dots n$ and any $h>0$
\begin{equation}\label{huvud11}
\int_{\Omega_j(h)}\frac{[f^*(t)-f^*(2t)]^p}{u_j(t)^p}\,dt\le c\left(\frac{\o_j(f;h)_p}{h}\right)^p,
\end{equation}
where $\Omega_j(h)=\{t>0: u_j(t)\ge h\}.$ Besides,
\begin{equation}\label{huvud111}
\sup_{t\in \Omega_j(h)}t^{1/p}[f^*(t)-f^*(2t)]/u_j(t)\le c \o_j(f;h)_p/h.
\end{equation}
\end{teo}
\begin{proof} By virtue of (\ref{moduli}), we may replace $f$ by one of its iterated rearrangements.
Furthermore, we
 may assume that $f$ is a positive continuous function on $\R_+^n$, strictly decreasing in each variable $x_k\in\R_+$.  In particular, our conditions imply that for any fixed $\widehat x_j\in \R^{n-1}_+$
\begin{equation}\label{huvud00}
\lim_{u\to+\infty} f(u, \widehat x_j)=0.
\end{equation}

Set for $t>0$
$$
E_t=\{x\in\R^n_+: f(x)>f^*(t)\}, \quad G_t=E_t\setminus E_{t/2}.
$$
Then $|G_t|=t/2.$ By Lemma \ref{projections},  there exist subsets
$$
G_t\equiv G_{t,0}\supset G_{t,1}\supset G_{t,2}\supset\cdots\supset G_{t,n}
$$
of type $F_\s$ such that $|G_{t,j}|=2^{-j-1}t$ and for any $j=1,\dots, n$
and any set $A\subset G_{t,j-1}$ of type $F_\s$ with $|A|\ge 2^{-j-1}t$ we have that
\begin{equation}\label{huvud2}
\mes_{n-1}\Pi_j(G_{t,j})\le \mes_{n-1}\Pi_j(A).
\end{equation}
Set
\begin{equation}\label{huvud22}
\mu_j(t)=2^{(n^2-1)/n}\mes_{n-1}\Pi_j(G_{t,j}).
\end{equation}
Applying the Loomis--Whitney inequality (\ref{lw}) to the set $G_{t,n}$, we have
$$
\prod_{j=1}^n \mes_{n-1}\Pi_j(G_{t,n})\ge |G_{t,n}|^{n-1}= \left(2^{-n-1}t\right)^{n-1}.
$$
Since $G_{t,n}\subset G_{t,j},$
then
$\mes_{n-1}\Pi_j(G_{t,n})\le\mes_{n-1}\Pi_j(G_{t,j})$. This  implies that
\begin{equation}\label{huvud3}
\prod_{j=1}^n \mu_j(t)\ge t^{n-1},\quad t>0.
\end{equation}
Set now
$$
u_j(t)=t/\mu_j(t)\quad(j=1,\dots, n).
$$
Then, by (\ref{huvud3}),
$$
\prod_{j=1}^n u_j(t)= t^n\left(\prod_{j=1}^n \mu_j(t)\right)^{-1}\le t\quad\mbox{for any}\quad t>0.
$$
Thus, the functions $u_j$ satisfy  (\ref{huvud1}).

Fix $h>0$ and $j\in\{1,\dots, n\}$.  Set
$$
f_h(y)=\frac1h\int_0^h f(y_j+u,\widehat y_j)\,du, \quad y=(y_j,\widehat y_j)\in \R^n_+.
$$
Since for a fixed $\widehat y_j$ the function $f(\cdot, \widehat y_j)$ decreases on $\R_+$,  we have that
\begin{equation}\label{huvud5}
f(y)\ge f_h(y)\quad\mbox{for any}\quad y\in \R^n_+.
\end{equation}
Let $\f_h=f-f_h.$ By (\ref{huvud5}), $\f_h\ge 0.$
Fix $t\in\Omega_j(h).$ Denote
$$
H_{t,j}=\{x\in G_{t,j-1}:\f_h(x)\le \f_h^*(2^{-j-1}t)\}.
$$
Then $|H_{t,j}|\ge 2^{-j-1}t$,  since $|G_{t,j-1}|=2^{-j}t$ and
$$
|\{x\in \R^n_+:\f_h(x)> \f_h^*(2^{-j-1}t)\}|=2^{-j-1}t.
$$
Thus, by (\ref{huvud2}) and (\ref{huvud22}), we have that
\begin{equation}\label{huvud7}
\mes_{n-1}\Pi_j(H_{t,j})\ge 2^{-(n^2-1)/n}\mu_j(t).
\end{equation}

 Fix $x=(x_j,\widehat x_j)\in H_{t,j}$. Since $H_{t,j}\subset G_t,$
 we have that $$f^*(t)<f(x)\le f^*(t/2).$$
 It follows from (\ref{huvud00}) that there exists $x_j'>x_j$ such that $x'=(x_j',\widehat x_j)\in
E_{3t}\setminus E_{2t},$ that is, $f^*(3t)<f(x')\le f^*(2t)$. Thus, taking into account (\ref{huvud5}), we obtain that
\begin{equation}\label{huvud4}
f^*(t)-f^*(2t)\le f(x)-f(x')\le \f_h(x)+f_h(x)-f_h(x').
\end{equation}
Since $x\in H_{t,j}$, we have that
$
0\le\f_h(x)\le \f_h^*(2^{-j-1}t).
$
Further, let
$S_t(\widehat x_j)$ denote the linear $\widehat x_j-$ section of the set $E_{3t}\setminus E_{t/2}.$
The function $f(\cdot,\widehat x_j)$ decreases on $\R_+.$
Therefore $[x_j, x_j']\subset S_t(\widehat x_j)$ and
$$
0\le f_h(x)-f_h(x')\le \int_{S_t(\widehat x_j)}\left|\frac{\partial f_h}{\partial x_j}(u, \widehat x_j)\right|\,du.
$$
Applying (\ref{huvud4}), we obtain that for any $\widehat x_j\in \Pi_j(H_{t,j})$
\begin{equation}\label{huvud44}
f^*(t)-f^*(2t)\le \f_h^*(2^{-j-1}t)+\int_{S_t(\widehat x_j)}\left|\frac{\partial f_h}{\partial x_j}(u, \widehat x_j)\right|\,du.
\end{equation}
Integrating inequality (\ref{huvud44}) with respect to $\widehat x_j$ over $\Pi_j(H_{t,j})$, we have
$$
\begin{aligned}
&f^*(t)-f^*(2t)\le \f_h^*(2^{-j-1}t)\\
&+\frac{1}{\mes_{n-1}\Pi_j(H_{t,j})}\int_{\Pi_j(H_{t,j})}\int_{S_t(\widehat x_j)}\left|\frac{\partial f_h}{\partial x_j}(u, \widehat x_j)\right|\,du\,d\widehat x_j.
\end{aligned}
$$
Observe that $H_{t,j}\subset G_t\subset E_{3t}\setminus E_{t/2}$ and therefore $$\Pi_j(H_{t,j})\subset \Pi_j(E_{3t}\setminus E_{t/2}).$$
Hence, using (\ref{huvud7}), we obtain
\begin{equation}\label{huvud77}
f^*(t)-f^*(2t)\le \f_h^*(2^{-j-1}t)+\frac{c}{\mu_j(t)}\int_{E_{3t}\setminus E_{t/2}}\left|\frac{\partial f_h}{\partial x_j}(x)\right|\,dx.
\end{equation}
Since $u_j(t)\ge h$ for all $t\in\Omega_j(h)$,
 estimate (\ref{huvud77}) implies that
$$
\begin{aligned}
&\left(\int_{\Omega_j(h)}\left(\frac{\mu_j(t)}{t}\right)^p[f^*(t)-f^*(2t)]^p\,dt\right)^{1/p}\le \frac1h\left(\int_0^\infty
\f_h^*(2^{-j-1}t)^p\,dt\right)^{1/p}\\
&+c\left(\int_0^\infty\left(\int_{E_{3t}\setminus E_{t/2}}\left|\frac{\partial f_h}{\partial x_j}(x)\right|\,dx\right)^p
\frac{dt}{t^p}\right)^{1/p}=J_1(h)+J_2(h).
\end{aligned}
$$
First, by (\ref{steklov1}),
\begin{equation}\label{huvud8}
J_1(h)\le \frac{2^{(j+1)/p}}{h}||\f_h||_p\le 2^{(j+1)/p}\frac{\o_j(f;h)_p}{h}.
\end{equation}
Further, by H\"older's inequality,
$$
J_2(h)^p\le c \int_0^\infty \int_{E_{3t}\setminus E_{t/2}}\left|\frac{\partial f_h}{\partial x_j}(x)\right|^p\,dx\,
\frac{dt}{t}.
$$
Set
$$
\Phi(u)=\int_{E_u}\left|\frac{\partial f_h}{\partial x_j}(x)\right|^p\,dx, \quad u>0.
$$
Then we have
$$
\begin{aligned}
J_2(h)^p\le &c\int_0^\infty[\Phi(3t)-\Phi(t/2)]\,\frac{dt}{t}=c\int_0^\infty\int_{t/2}^{3t}\Phi'(u)\,du\,\frac{dt}{t}\\
&\le c'\int_0^\infty \Phi'(u)\,du\le c'\int_{\R^n_+}\left|\frac{\partial f_h}{\partial x_j}(x)\right|^p\,dx.
\end{aligned}
$$
Applying  inequality (\ref{steklov2}), we get
$$
J_2(h)\le c \frac{\o_j(f;h)_p}{h}.
$$
Together with (\ref{huvud8}), this gives estimate (\ref{huvud11}).

Finally, by (\ref{huvud77}) and  H\"older's inequality, for any $t\in \Omega_j(h)$
$$
\begin{aligned}
&t^{1/p}[f^*(t)-f^*(2t)]/u_j(t)\\
&\le t^{1/p}\f_h^*(2^{-j-1}t)/u_j(t)
+ct^{1/p-1}\int_{E_{3t}\setminus E_{t/2}}\left|\frac{\partial f_h}{\partial x_j}(x)\right|\,dx\\
&\le \frac{2^{(j+1)/p}}{h}||\f_h||_p+c\left\|\frac{\partial f_h}{\partial x_j}(x)\right\|_p.
\end{aligned}
$$
Applying  (\ref{steklov1}) and  (\ref{steklov2}), we obtain (\ref{huvud111}).

We stress that the constants in (\ref{huvud11}) and (\ref{huvud111}) depend  on $p$ and $n$ only.

\end{proof}

\section{Limiting embeddings}

Applying Theorem \ref{Huvud}, we prove inequality (\ref{main2}).

\begin{teo}\label{MAIN_2} Let $1\le p<\infty, ~~p\le \t_j\le
\infty,$ and $0<\beta_j<1$ for all $j=1,...,n$ $(n\ge 2).$ Set
\begin{equation}\label{parameter}
\beta=n\left(\sum_{j=1}^n\frac1{\beta_j}\right)^{-1}, \quad
\t=\frac{n}{\beta}\left(\sum_{j=1}^n\frac1{\beta_j\t_j}\right)^{-1}.
\end{equation}
Assume that $1\le p<n/\beta.$  Let $q=np/(n-\beta p).$ Then there exists a constant $c>0$ such that for
any function
$$
f\in \bigcap_{j=1}^n B_{p,\t_j;j}^{\beta_j}(\R^n)
$$
we have that
\begin{equation}\label{main12}
||f||_{q,\t}\le c
\prod_{j=1}^n\left[(1-\beta_j)^{1/\t_j}||f||_{b_{p,\t_j;j}^{\beta_j}}\right]^{\beta/(n\beta_j)}.
\end{equation}
\end{teo}
\begin{proof}
Set $I=||f||_{q,\t}$; we assume that $I<\infty.$  Also, we assume that all $\t_j<\infty$ (this assumption is justified by (\ref{predel})).

Denote
$$
J=\left(\int_0^\infty t^{\t/q-1}[f^*(t)-f^*(2t)]^\t\,dt\right)^{1/\t}.
$$
By Minkowski's inequality,
$$
J\ge I-\left(\int_0^\infty t^{\t/q-1}f^*(2t)^\t\,dt\right)^{1/\t}= (1-2^{-1/q})I.
$$
Thus, $I\le (1-2^{-1/q})^{-1}J$, and it is sufficient to estimate $J$. Denote $$\f(t)=f^*(t)-f^*(2t).$$

Let $u_j \,\,(j=1,\dots, n)$ be the functions defined in Theorem \ref{Huvud}. We have
$$
\frac{\t}{q}-1=\frac{\t}{p}-1-\frac{\t\b}{n}.
$$
Thus, by (\ref{huvud1}),
$$
J^\t=\int_0^\infty t^{\t/p-1}\f(t)^\t t^{-\t\b/n}\,dt
$$
\begin{equation}\label{teor0}
\le \int_0^\infty t^{\t/p-1}\f(t)^\t\left(\prod_{k=1}^n u_k(t)\right)^{-\t\b/n}\,dt.
\end{equation}
 Denote
$$
r_k=n\t_k\b_k/(\t\b),\quad g_k(t)=(\f(t)t^{1/p-1/\t_k})^{\t_k/r_k}u_k(t)^{-\t\b/n}.
$$
We have that
$$
\sum_{k=1}^n\frac1{r_k}=1\quad\mbox{and}\quad \sum_{k=1}^n\frac{\t_k}{r_k}=\t.
$$
Therefore inequality (\ref{teor0}) can be written as
$$
J^\t\le\int_0^\infty\prod_{k=1}^n g_k(t)         \,dt.
$$
Applying H\"older's inequality with the exponents $r_k$, we have
\begin{equation}\label{teor1}
J^\t\le\prod_{k=1}^nJ_k^\t, \quad J_k=\left(\int_0^\infty t^{\t_k/p-1}\f(t)^{\t_k} u_k(t)^{-\t_k\b_k}\,dt\right)^{1/(\t r_k)}.
\end{equation}
Further,
$$
u_k(t)^{\t_k(1-\b_k)}=\t_k(1-\b_k)\int_0^{u_k(t)} h^{\t_k(1-\b_k)-1}\,dh.
$$
Thus, applying Fubini's theorem, we have
$$
\begin{aligned}
J_k^{\t r_k}&=\t_k(1-\b_k)\int_0^\infty t^{\t_k/p-1}\left(\frac{\f(t)}{u_k(t)}\right)^{\t_k}\int_0^{u_k(t)} h^{\t_k(1-\b_k)-1}\,dh\,dt\\
&=\t_k(1-\b_k)\int_0^\infty h^{\t_k(1-\b_k)-1} \int_{\Omega_k(h)}t^{\t_k/p-1}\left(\frac{\f(t)}{u_k(t)}\right)^{\t_k}\,dt\,dh,
\end{aligned}
$$
where $\Omega_k(h)=\{t>0: u_k(t)\ge h\}.$

Applying (\ref{huvud111}) and (\ref{huvud11}), and taking into account that $\t_k\ge p,$ we obtain
$$
\begin{aligned}
&\int_{\Omega_k(h)}t^{\t_k/p-1}\left(\frac{\f(t)}{u_k(t)}\right)^{\t_k}\,dt\\
&\le c^{\t_k-p} \left(\frac{\o_k(f;h)_p}{h}\right)^{\t_k-p}\int_{\Omega_k(h)}\left(\frac{\f(t)}{u_k(t)}\right)^p\,dt\le c_1^{\t_k}\left(\frac{\o_k(f;h)_p}{h}\right)^{\t_k}
\end{aligned}
$$
(the constant $c_1$ depends only on $p$ and $n$). It follows that
$$
J_k^{\t r_k}\le c_1^{\t_k}\t_k(1-\b_k)\int_0^\infty h^{-\t_k\b_k}\o_k(f;h)_p^{\t_k}\,\frac{dh}{h}.
$$
Applying (\ref{teor1}), we obtain inequality (\ref{main12}) with the constant
$$
c=c_1/(1-2^{-1/q})\prod_{k=1}^n\t_k^{\b/(\t_kn\b_k)}\le 2c_1/(1-2^{-1/q})
$$
(we have used the inequality  $s^{1/s}<2,\,\, s\ge 1$).

We stress that the constant is bounded with respect to $\t_k.$ Due to (\ref{predel}), this enables us to drop the assumption
$\t_k<+\infty.$
\end{proof}

We observe that in \cite{K9} we obtained estimate of modified Lorentz norms (defined in terms of iterated rearrangements) with the same right-hand side as in (\ref{main12}). In the case $\t\le q$ this estimate is stronger than (\ref{main12}); however, for
$\t>q$ the relation is opposite. We shall discuss these results in Section  5 below.

Note that the use of Theorem \ref{MAIN_2} together with Lemma \ref{K-L} enables us to derive some limiting embeddings for Lipschitz classes. We give a short description of this result.

Assume that
$0<\a\le 1$,  $1\le
p<\infty$, and $1\le k\le n.$ Denote by $\Lambda_{p;k}^\a(\R^n)$
the class of all functions $f\in L^p(\R^n)$ such that
\begin{equation*}
||f||_{\l_{p;k}^\a}=
\sup_{\d>0}\frac{\o_k(f;\d)_p}{\d^{\a}}<\infty.
\end{equation*}
Clearly, $||f||_{\l_{p;k}^\a}=||f||_{b_{p, \infty;k}^\a}$ if $\a<1.$

Let $1\le p <\infty, ~~0<\a_k\le 1,$ and
$$
\a= n\left(\sum_{k=1}^n\frac1{\a_k}\right)^{-1}<\frac np.
$$
 Let $\nu$ be the number of $\a_j$ that are
equal to 1. Set  $q^*=np/(n-\a p)$ and $s=np/(\nu \a)$. Applying Theorem \ref{MAIN_2} and Lemma \ref{K-L}, we obtain that for any function
$$
f\in \bigcap_{k=1}^n
\Lambda_{p;k}^{\a_k}(\R^n)
$$
it holds the following inequality
\begin{equation*}
||f||_{q^*,s}\le c
\prod_{k=1}^n||f||_{\l_{p;k}^{\a_k}}^{\a/(n\a_k)}.
\end{equation*}

For the proofs, references, and other details, we refer to  \cite[\S 8]{K9}.

\vskip 8pt

\section{Appendix}

In this section we consider  estimate of modified Lorentz norms  obtained in \cite{K9}.

For any $x=(x_1,\dots,x_n)\in\R^n_+$, set
$
\pi(x)=\prod_{k=1}^n x_k.
$

 Let $0<p, r<\infty$ and let $\sigma \in \mathcal P _n\
(n\ge 2).$   Denote by $\mathcal L^{p,r}_{\sigma}(\Bbb R ^n)$ the
class of all functions $f\in S_0(\Bbb R ^n)$ such that
$$
\| f\| _{p,r;\sigma}= \left( \int_{\Bbb R _+^n}\left[
\pi(t)^{1/p}\mathcal R _\sigma f(t)\right] ^r\,
\frac{dt}{\pi(t)}\right) ^{1/r}<\infty
$$
(see \cite{Bl}). The choice of a permutation $\s$ is essential. We
also set
$$
\mathcal L^{p,r}(\Bbb R^n)=\bigcap _{\sigma\in \mathcal P
_n}\mathcal L^{p,r}_{\sigma}(\Bbb R ^n),\quad \| f\|_{\mathcal
L^{p,r}}=\sum _{\sigma\in \mathcal P _n}\| f\| _{p,r;\sigma}.
$$

The relations between $L^{p,r}$- and $\mathcal L^{p,r}_{\sigma}$-norms
are  the following.

{\it Let $f\in S_0(\Bbb R ^n).$ Then for any}
  $\sigma \in \mathcal P _n$
\begin{equation}\label{yats1}
\| f\| _{p,r}\le c\| f\|
_{p,r;\sigma}\quad\text{{\it if}}\quad 0<r\le p<\infty
\end{equation}
{\it and}
\begin{equation}\label{yats2}
\| f\| _{p,r;\sigma}\le  c\| f\|
_{p,r}\quad\text{{\it if}}\quad 0<p<r<\infty.
\end{equation}

These inequalities were first obtained in \cite{Ya}. For the sharp constants, see \cite{BKPS}, \cite{K2012}.

 We have the following theorem \cite{K9}.

\begin{teo}\label{MAIN1000} Let $1\le p<\infty, ~~p\le \t_j\le
\infty,$ and $0<\beta_j<1$ for all $j=1,...,n,$ where $n\ge 2.$ Set
$$
\beta=n\left(\sum_{j=1}^n\frac1{\beta_j}\right)^{-1}, \quad
\t=\frac{n}{\beta}\left(\sum_{j=1}^n\frac1{\beta_j\t_j}\right)^{-1}.
$$
Assume that $1\le p<n/\beta.$  Let $q=np/(n-\beta p).$ Then there exists a constant $c>0$ such that for
any function
$$
f\in \bigcap_{j=1}^n B_{p,\t_j;j}^{\beta_j}(\R^n)
$$
and any $\sigma \in \mathcal P _n$
we have that
\begin{equation}\label{main1000}
||f||_{\mathcal L_\s^{q,\t}}\le c
\prod_{j=1}^n\left[(1-\beta_j)^{1/\t_j}||f||_{b_{p,\t_j;j}^{\beta_j}}\right]^{\beta/(n\beta_j)}.
\end{equation}
\end{teo}

By virtue of (\ref{yats1}) and (\ref{yats2}), for $\t\le q$ Theorem \ref{MAIN_2} follows from Theorem \ref{MAIN1000}, and for $\t>q$ the relation is
opposite.

The sketch of the proof of  Theorem \ref{MAIN1000} was given in \cite{K9}. However, there were some inaccuracies, which have brought to a gap in the proof. Therefore we give here a complete  proof of this theorem.

Denote
\begin{equation}\label{Q}
Q(x)=Q_n(x)=[x_1/2,x_1]\times\cdots\times[x_n/2,x_n],\quad x\in \R^n_+.
\end{equation}
Further, for a nonnegative function $\f\in L_{loc}(\R^n_+),$ set
\begin{equation}\label{T}
T\f(x)=\frac1{|Q(x)|}\int_{Q(x)}\f(u)\,du.
\end{equation}

\begin{lem}\label{Oper} Let $\f\in L_{loc}(\R^n_+)$ be a nonnegative function. Then for any $r\ge 1$ and any $\a\in \R$
\begin{equation}\label{oper}
\int_{\R^n_+}(T\f(x))^r\pi(x)^\a\,dx\le 2^{\max(1,\a)n}\int_{\R^n_+}\f(x)^r\pi(x)^\a\,dx.
\end{equation}
\end{lem}
\begin{proof} 
Applying H\"older's inequality and Fubini's theorem, we have
$$
\begin{aligned}
&\int_{\R^n_+}(T\f(x))^r\pi(x)^\a\,dx
\le 2^n\int_{\R^n_+}\pi(x)^{\a-1}\int_{Q(x)}\f(u)^r\,du\,dx\\
&=2^n\int_{\R^n_+}\f(u)^r\int_{Q(2u)}\pi(x)^{\a-1}\,dx\,du
\le 2^{\max(1,\a)n}\int_{\R^n_+}\f(u)^r\pi(u)^\a\,du.
\end{aligned}
$$

\end{proof}

Observe that for any $\f\in \mathcal{M}_{dec}(\R^n_+)$
\begin{equation}\label{monotone}
\f(x)\le T\f(x)\quad \mbox{for all}\quad x\in\R^n_+.
\end{equation}

The following lemma was proved in  \cite{K9}.

\begin{lem}\label{Iter} Let $1\le p<\infty$ and $n\ge 2.$ Assume
that $f\in L^p(\mathbb R_+^n)\cap \mathcal{M}_{dec}(\R^n_+)$.
 Let $\mu>1.$ Then for any $1\le k\le n$ and any
$h>0,$
$$
\left(\int_{R_+^{n-1}}\int_h^\infty u^{-p}~[f(u,\widehat
t_k)-f(\mu u,\widehat t_k)]^pdud\widehat t_k\right)^{1/p}
$$
\begin{equation}\label{iter}
\le 4\mu\frac{\o_k(f;h)_p}{h}.
\end{equation}
\end{lem}

{\it Proof of Theorem \ref{MAIN1000}}. Let $\s=\{1,...,n\}.$ Denote $F(x)=\mathcal R_\s
f(x), ~~x\in \R^n_+.$ We may assume that the left-hand side of (\ref{main1000}) is finite and that all $\t_j<\infty$.

For each $k=1,\dots,n$, denote
\begin{equation*}
A_k=\{x\in\R^n_+: F(x)\le 2F(\mu x_k, \widehat x_k)\},\quad f_k(x)=F(\mu x_k, \widehat x_k)\chi_{A_k}(x),
\end{equation*}
where a sufficiently big number $\mu$ will be chosen below.
Set also
\begin{equation}\label{7.000}
E=\R^n_+\setminus \left(\bigcup_{k=1}^n A_k\right) \quad\mbox{and}\quad g(x)=F(x)\chi_E(x).
\end{equation}
Then
\begin{equation}\label{7.1}
F(x)\le g(x) +2 \sum_{k=1}^n f_k(x)\quad\text{for
all}\quad x\in \R^n_+.
\end{equation}

 Set
$$
I = \left(\int_{\R^n_+} F(x)^\t\pi(x)^{\t/q-1}\,dx\right)^{1/\t}.
$$

Using operator $T$ defined by (\ref{T}), we have by (\ref{7.1})
\begin{equation}\label{7.11}
 TF(x)\le Tg(x)+2 \sum_{k=1}^n Tf_k(x).
 \end{equation}
Applying Lemma \ref{Oper}, we have for any $k=1,\dots,n$
$$
\begin{aligned}
&\left(\int_{\R^n_+}(Tf_k(x))^\t\pi(x)^{\t/q-1}\,dx\right)^{1/\t}\\
&\le c_1\left(\int_{\R^n_+}F(\mu x_k,\widehat x_k)^\t\pi(x)^{\t/q-1}\,dx\right)^{1/\t}=\frac{c_1}{\mu^{1/q}}I,
\end{aligned}
$$
where $c_1=2^{n\max(\t/q-1, 1)/\t}.$ From here,  taking into account (\ref{monotone}) and applying (\ref{7.11}), we obtain
$$
I\le
\left(\int_{\R^n_+}(Tg(x))^\t\pi(x)^{\t/q-1}\,dx\right)^{1/\t} + \frac{2nc_1}{\mu^{1/q}}I.
$$
Setting $\mu=(4nc_1)^{q},$ we get
\begin{equation}\label{7.2}
I\le 2\left(\int_{\R^n_+}(Tg(x))^\t\pi(x)^{\t/q-1}\,dx\right)^{1/\t}.
\end{equation}

Further, denote
\begin{equation}\label{7.3}
g_k(x)=[F(x)-F(\mu x_k,\widehat x_k)]\chi_E(x)
\end{equation}
By the definition of $A_k$ and (\ref{7.000}), for any $k=1,\dots,n$
\begin{equation}\label{7.4}
g(x)\le 2 g_k(x),\,\, x\in \R^n_+.
\end{equation}
Thus, for any $k=1,\dots,n$
\begin{equation}\label{7.5}
Tg(x)\le 2 Tg_k(x),\,\, x\in \R^n_+.
\end{equation}
By H\"older's inequality,
$$
\begin{aligned}
&Tg_k(x)\le \frac 1{|Q(x)|}\int_{Q(x)}[F(u)-F(\mu u_k,\widehat u_k)]\,du\\
&\le|Q(x)|^{-1/p}\left(\int_{Q(x)}[F(u)-F(\mu u_k,\widehat u_k)]^p\,du\right)^{1/p}.
\end{aligned}
$$
Since $u_k\le x_k$ for $u\in Q(x),$ and $F\in \mathcal M_{dec}(\R^n_+),$ we obtain that
$$
\begin{aligned}
Tg_k(x)
&\le 2^{n/p}\pi(x)^{-1/p}\left(\int_{\R^n_+}[F(u)-F(u+\mu x_ke_k)]^p\,du\right)^{1/p}\\
&\le c_n\mu\pi(x)^{-1/p}\o_k(F;x_k)_p.
\end{aligned}
$$
 Taking into account (\ref{7.5}) and applying inequality (\ref{moduli}), we have that for any $x\in\R^n_+$ and any $k=1,\dots,n$
\begin{equation}\label{7.6}
\pi(x)^{1/p}Tg(x)\le c_n\mu\o_k(f;x_k)_p.
\end{equation}

 Denote
$$
r_k=n\t_k\b_k/(\t\b),\quad \l_k(x)=\left(Tg(x)\pi(x)^{1/p-1/\t_k}\right)^{\t_k/r_k}x_k^{-\t\b/n}.
$$
We have that $\t/q-1=\t/p-1-\t\b/n$, 
\begin{equation}\label{7.70}
\sum_{k=1}^n1/r_k=1,\quad\mbox{and}\quad \sum_{k=1}^n\t_k/r_k=\t.
\end{equation}
Therefore inequality (\ref{7.2}) can be written in the form
$$
I\le 2\left(\int_{\R^n_+}\prod_{k=1}^n \l_k(x)\,dx\right)^{1/\t}.
$$
Hence,  applying H\"older's inequality with the exponents $r_k$, we obtain
\begin{equation}\label{7.8}
I\le 2\prod_{k=1}^n I_k,
\quad
I_k=\left(\int_{\R^n_+}x_k^{-\t_k\b_k}\pi(x)^{\t_k/p-1}Tg(x)^{\t_k}\,dx\right)^{1/(\t r_k)}.
\end{equation}

Recall that $\t_k\ge p.$ Thus, using estimate  (\ref{7.6}), we have
$$
I_k^{\t r_k}\le (c_n\mu)^{\t_k-p}\int_{\R^n_+}x_k^{-\t_k\b_k}\o_k(f;x_k)_p^{\t_k-p} Tg(x)^p\,dx.
$$
Setting
$$
\f_k(z)=\left(\int_{\R^{n-1}_+} Tg_k(z,\widehat x_k)^p\,d\widehat x_k\right)^{1/p},\,\, z\in \R_+,
$$
and applying (\ref{7.5}), we obtain
$$
I_k^{\t r_k}\le (2c_n\mu)^{\t_k-p}\int_0^\infty z^{\t_k(1-\b_k)}\frac{\o_k(f;z)_p^{\t_k-p}\f_k(z)^p}{z^{\t_k}}\,dz.
$$
Observe that
$$
z^{\t_k(1-\b_k)}=\t_k(1-\b_k)\int_0^z h^{\t_k(1-\b_k)-1}\,dh.
$$
Thus, using Fubini's theorem, we have
$$
I_k^{\t r_k}\le (2c_n\mu)^{\t_k-p}\t_k(1-\b_k)\int_0^\infty h^{\t_k(1-\b_k)-1}\int_h^\infty\frac{\o_k(f;z)_p^{\t_k-p}\f_k(z)^p}{z^{\t_k}}\,dz\,dh.
$$
Using inequality (\ref{d3}), we get
\begin{equation}\label{7.10}
I_k^{\t r_k}\le (4c_n\mu)^{\t_k-p}\t_k(1-\b_k)\int_0^\infty \frac{\o_k(f;h)_p^{\t_k-p}}{h^{\t_k\b_k-p}}\int_h^\infty\frac{\f_k(z)^p}{z^{p}}\,dz\,\frac{dh}{h}.
\end{equation}

Now we estimate $\f_k(z).$ First, we have by H\"older's inequality
$$
Tg_k(x)^p\le \frac1{|Q(x)|}\int_{Q(x)} g_k(u)^p\,du.
$$
Using this inequality, we easily obtain that
$$
\f_k(x_k)^p=\int_{\R^{n-1}_+} (Tg_k(x))^p\,d\widehat x_k
\le\frac{2^n}{x_k}\int_{x_k/2}^{x_k}\psi(u_k)\,du_k,
$$
where
$$
\psi(u_k)=\int_{\R^{n-1}_+}[F(u)-F(\mu u_k,\widehat u_k)]^p\, d\widehat u_k.
$$
Thus,
$$
\begin{aligned}
&\int_h^\infty\frac{\f_k(z)^p}{z^{p}}\,dz\le 2^n\int_h^\infty \frac{dz}{z^{p+1}}\int_{z/2}^z\psi(u_k)\,du_k\\
&\le 2^n \int_{h/2}^\infty\frac{\psi(u_k)}{u_k^{p}}\,du_k\le2^n \int_{h/2}^\infty\frac{\psi(u_k)}{u_k^{p}}\,du_k\int_{\R^{n-1}_+}[F(u)-F(\mu u_k,\widehat u_k)]^p\, d\widehat u_k.
\end{aligned}
$$
Applying  Lemma \ref{Iter} and inequality (\ref{moduli}), we get
\begin{equation}\label{fi}
\left(\int_h^\infty\frac{\f_k(z)^p}{z^{p}}\,dz\right)^{1/p}\le c_n\mu\frac{\o_k(F;h)_p}{h} \le c_n'\mu\frac{\o_k(f;h)_p}{h}.
\end{equation}
Using this estimate and (\ref{7.10}), we obtain that for any $k=1,\dots,n$
$$
I_k^{\t r_k}\le  (c_n\mu)^{\t_k}\t_k(1-\b_k)\int_0^\infty \frac{\o_k(f;h)^{\t_k}_p}{h^{\t_k\b_k}}\,\frac{dh}{h},
$$
where $c_n$ depends only on $n$ and $\mu\le 2^{(n/\min(q,\t)+2)q}\le 2^{(n+2)q}.$ Thus,
$$
I_k\le (c_n\mu)^{\t_k/(\t r_k)}\left[\t_k(1-\b_k)||f||^{\t_k}_{b_{p,\t_k;k}^{\beta_k}}\right]^{1/(\t r_k)}.
$$
By virtue of (\ref{7.8}),  (\ref{7.70}),  equality $\t r_k=n\b_k\t_k/\b,$ and inequality $s^{1/s}<2 \,\,(s\ge 1)$, this implies that
\begin{equation}\label{mainII}
||f||_{q,\t;\s}\le
2c_n\mu\prod_{k=1}^n\left[(1-\beta_k)^{1/\t_k}||f||_{b_{p,\t_k;k}^{\b_k}}\right]^{\beta/(n\beta_k)}.
\end{equation}

\end{document}